\documentclass[preprint,11pt]{article}
\usepackage[pagewise]{lineno}
\usepackage{amssymb}
\usepackage{latexsym,amsfonts,amssymb,amsmath,amsthm}
\usepackage{graphicx}
\usepackage{cite}
\usepackage{amsthm}
\usepackage{amsfonts}
\usepackage{xr}
\usepackage{graphicx}
\usepackage{amsmath}
\usepackage{color}
\usepackage{amssymb}
\usepackage{mdwlist}

\newtheorem{theorem}{Theorem}[section]
\newtheorem{definition}{Definition}[section]

\newtheorem{lemma}[theorem]{Lemma}
\newtheorem{rem}[theorem]{Remark}

\numberwithin{equation}{section}

  \def \G{\Gamma}

\makeatletter 
\@addtoreset{equation}{section}
\makeatother  

\newcommand\norm[1]{\lVert#1\rVert}
\newcommand\abs[1]{\lvert#1\rvert}

\newcommand\RR{\ensuremath{\mathbb{R}}}

\def\tint{\text{Int}}
\newcommand{\ucite}[1]{\cite{#1}}

\begin{document}

\title{Semiflows strongly focusing monotone with respect to high-rank cones: II. Pseudo-ordered principle}

\author{Lirui Feng\thanks{School of Mathematical Sciences, University of Science and Technology of China, Hefei, Anhui, 230026, People’s Republic of China (ruilif@ustc.edu.cn). The author is supported by NSF of China No.12101583 and No.12331006.}}
\date{}

\maketitle

\begin{abstract}
We consider a semiflow strongly focusing monotone with respect to a cone of rank $k$ on a Banach space. We prove that the omega-limit set of a pseudo-ordered semiorbit is ordered, which is called as pseudo-ordered principle. Based on this principle, we obtain the solid Poincar\'{e}-Bendixson theorem with the rank $k=2$, that is, the omega-limit set of a pseudo-ordered semiorbit is either a nontrivial periodic orbit or a set consisting of equilibria with their potential connected orbits. The generic solid dynamics theorem with a general rank $k$ and the generic solid Poincar\'{e}-Bendixson theorem with the rank $k=2$ are also obtained.
\end{abstract}

\section{Introduction}

In this paper, we investigate the total-ordering property of the omega-limit set of a pseudo-ordered semiorbit for the semiflow $\Phi_t$ strongly focusing monotone with respect to a cone $C$ of rank $k$ on an infitnite dimensional Banach space $X$, and improve some related properties of its. A cone $C$ of rank $k$ (abbr. $k$-cone) is a closed subset of $X$ such that $\mathbb{R}\cdot C\subset C$ and it contains a subspace of dimension $k$ but no subspace of higher dimension, which is introduced by Krasnosel'skij, Lifshits and Sobolev \cite{K-L-S}, and also by Fusco and Oliva \cite{F-O-1,F-O-2} for the case of finite dimensional spaces. Roughly speaking, a semiflow $\Phi_t$ strongly focusing monotone with respect to a cone $C$ of rank $k$, is a smooth semiflow strongly monotone with respect to $C$ possessing the strongly focusing monotonicity (see Definition \ref{strongly focusing monotone}(ii)), that provides clues on local behaviors around invaraint sets of the semiflow by strongly focusing operators. We refer to the series of wroks in \cite{F-W-W1,F-W-W2,F-W-W3,F-W-W4,F} for the strong monotonicity with respect to $k$-cones, and the series of poineering works in \cite{Hir-1,Hir-2,Hir-3,Hir-4,Hir-5,Hir-6,Hir-7} for the strong monotonicity with respect to convex cones. A strongly focusing operator with respect to a $k$-cone originated from Krasnosel'skij et al. \cite{K-L-S} to prove a Krein-Rutman type theorem with respect to a k-cone for a single operator, and also from Lian and Wang \cite{LW2} to investigate the relationship between Multiplicative Ergodic theorem and Krein-Rutman type theorem for random linear dynamical systems. The notable feature of a strongly focusing operator $R$ with respect to $C$ is that the operator $R$ maps $C\setminus\{0\}$ into the interior of $C$ such that unit vectors in $R C$ are uniformly separated from the boundary of $C$. This feature also makes a semiflow strongly focusing monotone with respect to a $k$-cone on an infinite dimensional Banach space be a kind of natrual extension of a flow strongly monotone with respect to a $k$-cone on $\mathbb{R}^n$ (See flows in \cite{F-W-W2}). 

The research on total-ordering property of the entire omega-limit (abbr. $\omega$-limit) set of a pseudo-ordered orbit of a flow strongly monotone with respect to a $k$-cone is proposed by Sanchez via an open problem \cite[p.1984]{San09}. It is hereafter referred to as 

\noindent{\bf Pseudo-ordered principle:}  The $\omega$-limit set of a precompact pseudo-ordered semiorbit is ordered.
 
\noindent Here, $x\in X$ and $y\in X$ are called {\it ordered} if $x-y\in C$, denoted by $x\thicksim y$; otherwise, they are called {\it unordered}, denoted by $x\rightharpoondown y$. Moreover, a subset $B\subset X$ is called {\it ordered} if $x\thicksim y$ for any $x,y\in B$; and is called {\it unordered} if it is not a singleton and $x\rightharpoondown y$ for any distinct $x,y\in B$. In Sanchez's works, he used the $C^1$-closed lemma to prove that each orbit in the $\omega$-limit set of a pseudo-ordered orbit is ordered (see \cite[Theorem 2]{San09}). Feng, Wang, and Wu extended this result to a semiflow strongly monotone with respect to a $k$-cone $C$ on an infinite dimensional Banach space $X$ without the smoothness assumption (see \cite[Theorem A]{F-W-W1}). By creative utilizing the topological properties of continuous semiflows, they obtained the order-trichotomy of the $\omega$-limit set $\Omega$ of a pseudo-ordered semiorbit $O^+(x)$ (see \cite[Theorem B]{F-W-W1} and also Lemma \ref{ordered trichotomy}), that is, either,

\begin{itemize}
\item[(a)] $\Omega$ is ordered; or, 
\item[(b)] $\Omega$ is an unordered set consisting of equilibria; or otherwise, 
\item[(c)] $\Omega$ possesses a certain ordered homoclinic property. 
\end{itemize}

\noindent For a flow strongly monotone with respect to a $k$-cone $C$ on $\mathbb{R}^n$, Feng, Wang, and Wu improved the order-trichotomy, and prove the pseudo-ordered principle by a series of technical lemmas based on the smoothness of the flow and the compactness of the closed unit ball in $\mathbb{R}^n$ (see \cite[Theorem A]{F-W-W4}). However, the closed unit ball in an infinite dimensional Banach space is not compact, and the inverse of the $x$-derivative $D_x\Phi_t$ of the semiflow $\Phi_t$ may not exist for each $(x,t)\in X\times \mathbb{R}^+.$ Therefore, we must conquer many difficulties to investigate the total-ordering property of the $\omega$-limit set $\Omega$ of a pseudo-ordered semiorbit $O^+(x)$ for a semiflow $\Phi_t$ strongly focusing monotone with respect to a $k$-cone $C$ on a Banach space $X$. 

In order to prove the pseudo-ordered principle for the semiflow $\Phi_t$, the idea is to prove that the alternatives (b) and (c) in the order-trichotomy for the $\omega$-limit set $\Omega$ of a pesudo-ordered semiorbit don't occur. The most important tool is the $k$-exponential separation of the linear skew-product semiflow $(\Phi_t, D\Phi_t)$ on $\Omega\times X$, which ensures that $(\Phi_t, D\Phi_t)$ on $\Omega\times X$ possesses an (unique) invariant bundle decomposition $\Omega\times X=\Omega\times (E_y)\oplus\Omega\times (F_y)$ such that the action of $(\Phi_t, D\Phi_t)$ on $\Omega\times (E_y)$ dominates the one on $\Omega\times (F_y)$. 

A notable feature in alternatives (b) and (c) is that there exists an equilibrium $p$ in $\Omega$. This inspires us to use the $k$-exponential separation to obtain the growth exponent of the actions of $D_p\Phi_t$ on its invariant subspaces $E_p$ and $F_p$ respectively. In the general case of infinite dimensional Banach spaces, our approach to realize the aim (see Lemma \ref{invariant subspace}) is analyzing the spectrum of $D_p\Phi_t$ on $E_p$ and $F_p$, which is instead of utilizing the perron theorem (see \cite[Theorem 1]{F-O-2}) in the case of finite dimensional spaces. 

To overcome the difficulty caused by the nonexistence of the inverse of $D_y\Phi_t$ for $(y,t)\in X\times \mathbb{R}^+$, we make new estimates on the action of $D_p\Phi_t, (p,t)\in X\times \mathbb{R}^+ $ to obtain that the sign of the growth exponent of the action of $D_p\Phi_t$ on $E_p$ is positive with the assumption that one of the alternatives (b) and (c) holds. 

Compared with the flows on a finite dimensional space, the lack of compactness of the closed unit ball in an infinite dimensional Banach space results that the positive sign of the growth exponent of the action of $D_{\tilde{p}}\Phi_t$ on $E_{\tilde{p}}$ for a general equilibrium ${\tilde{p}}$, can not imply the behavior of the semiorbits separating away a certain small neighborhood $\mathcal{V}_{\tilde{p}}$ of $\tilde{p}$ with its initial point  in $\mathcal{V}_{\tilde{p}}$, ordered and different with ${\tilde{p}}$. For convenience, this behavior is called as {\it the separating behavior}. Even if the strongly focusing monotonicity of the semiflow $\Phi_t$ brings some useful information, we can only confirm that the separating behavior is local and potentially temporary with certain separating estimates (see Lemma \ref{p-notin-omega.y}). To deal with the local and potentially temporary characteristic of this separating behavior for the semiflow $\Phi_t$ on an infinite dimensional Banach space, we combinate the strongly focusing monotonicity of $\Phi_t$ and the compactness of the $\omega$-limit set $\Omega$ to do delicate arguments for the special equilibrium $q$ in $\Omega$, and prove that the alternative (c) does not occur (see Lemma \ref{Lemma for iii}). The existence of a positive lower bound of limit superior of the distance between two certain points in semiorbits with different initial points (see Lemma \ref{L:instabilty}) are used in the arguments to prove that the alternative (b) does not occur (see Lemma \ref{Lemma for ii}). Together with Lemma \ref{Lemma for ii} and \ref{Lemma for iii}, we prove the pseudo-ordered principle for the semiflow $\Phi_t$. 

It should be pointed out that the arguments on the solid Poincar\'{e}-Bendixson theorem for flows on a finite dimensional space (see \cite[Theorem B]{F-W-W4}) are still effective for the semiflow $\Phi_t$ strongly focusing monotone with respect to a $k$-cone $C$ on an infinite dimensional Banach space. We also improve the generic dynamics of the semiflow $\Phi_t$ (see \cite[Theorem A]{F}) and obtain the generic solid dynamics theorem. The generic solid Poincar\'{e}-Bendixson theorem for the semiflow $\Phi_t$ is obtained when the rank of $C$ is $k=2$.

This paper is organized as follows. In section 2, we introduce notations and main results. In section 3, we present and prove preliminary lemmas. In section 4, we prove our main results.

\section{Notations and Main results}

We start with basic notations and definitions on semiflows strongly focusing monotone with respect to high-rank cones. Let $(X,\norm{\cdot})$ be a Banach space equipped with a norm $\norm{\cdot}$. A {\it semiflow} on $X$ is a continuous map $\Phi:\mathbb{R}^+\times X\to X$ with $\Phi_0={\rm Id}$ and $\Phi_t\circ\Phi_s=\Phi_{t+s}$ for $t,s\ge 0$. Here, $\Phi_t(\cdot)=\Phi(t,\cdot)$ for $t\in\mathbb{R}^+$, and ${\rm Id}$ is the identity map on $X$. A semiflow $\Phi_t$ on $X$ is called {\it $C^{1,\alpha}$-smooth} if $\Phi|_{\mathbb{R}^+\times X}$ is a $C^{1,\alpha}$-map (a $C^1$-map with a locally $\alpha$-H\"{o}lder derivative) with $\alpha\in (0,1]$. The derivative of $\Phi_t$ with respect to $x$, at $(t,x)$, is denoted by $D_x\Phi_t$.

Let $x\in X$,  the {\it positive semiorbit of $x$} is denoted by $O^+(x)=\{\Phi_t(x):t\ge 0\}$. A {\it negative semiorbit}
(resp. {\it full-orbit}) of $x$ is a continuous function $\psi:\mathbb{R}^-=\{t\in \mathbb{R}: t\le 0\}\to X$ (resp. $\psi:\mathbb{R}\to X$) such that $\psi(0)=x$, and for any $s\le 0$ (resp. $s\in \mathbb{R}$), $\Phi_t(\psi(s))=\psi(t+s)$ holds for $0\le t\le -s$ (resp. $0\le t$). Clearly, if $\psi$ is a negative semiorbit of $x$, then $\psi$ can be extended to a full-orbit, that is, $\tilde{\psi}(t)=\psi(t)$ for $t\leq 0$, and $\tilde{\psi}(t)=\Phi_t(x)$ for any $t\geq 0$. On the other hand, any full-orbit of $x$ when restricted on $\mathbb{R}^-$ is a negative semiorbit of $x$. Since $\Phi_t$ is just a semiflow, a negative semiorbit of $x$ may not exist, and it is not necessary to be unique even if one exists.

An {\it equilibrium} (also called {\it a trivial orbit}) is a point $x$ for which $O^+(x)=\{x\}$. Let $E$ be the set of all equilibria of $\Phi_t$. A nontrivial semiorbit $O^+(x)$ is said to be a {\it periodic orbit} if $\Phi_T(x)=x$ for some $T>0$; and it is said to be a {\it $T$-periodic orbit} if there is a $T>0$ such that $\Phi_T(x)=x$ and $\Phi_t(x)\neq x$ for any $t\in(0,T)$, where $T$ is called the {\it minimal period} of $O^+(x)$. A subset $\Sigma\subset X$ is called {\it positively invariant with respect to $\Phi_t$ {\rm (}for short, positively invariant{\rm)}} if $\Phi_{t}(\Sigma)\subset \Sigma$ for any $t\in \mathbb{R}^+$; and it is called {\it invariant} if $\Phi_{t}(\Sigma)=\Sigma$ for any $t\in \mathbb{R}^+$. Clearly, for any $x\in \Sigma$, there exists a negative semiorbit of $x$, provided that $\Sigma$ is invariant. Let $\Sigma\subset X$ be an invariant set. $\Phi_t$ is said to {\it admit a flow extension on $\Sigma$}, if there is a flow $\tilde{\Phi}_t$ such that $\tilde{\Phi}_t(x)=\Phi_t(x)$ for any $x\in \Sigma$ and $t\ge 0.$ The {\it omega-limit {\rm(}abbr. $\omega$-limit{\rm)} set} $\omega(x)$ of $x\in X$ is defined by $\omega(x)=\cap_{s\ge 0}\overline{\cup_{t\ge s}\Phi_t(x)}$. If $O^+(x)$ is precompact, then $\omega(x)$ is nonempty, compact, connected, and invariant. Given a subset $D\subset X$, {\it the positive semiorbit $O^+(D)$ of $D$} is defined as $O^+(D)=\bigcup\limits_{x\in D}O^+(x)$. A subset $D$ is called {\it $\omega$-compact} if $O^+(x)$ is precompact for each $x\in D$, and $\bigcup\limits_{x\in D}\omega(x)$ is precompact. Clearly, $D$ is $\omega$-compact provided by the compactness of $\overline{O^+(D)}$. Given a negative semiorbit $\psi$ of $x$, it is also denoted by $O_{b}^-(x)$ in order to emphasize its negativity, initial point and image in the phase space $X$. If $O^{-}_{b}(x)$ is precompact, the {\it alpha-limit {\rm(}abbr. $\alpha$-limit{\rm)} set} $\alpha_{b}(x)$ of $O^{-}_{b}(x)$ is defined as $\alpha_{b}(x)=\cap_{s\geq 0}\overline{\cup_{t\geq s}\psi(-t)}$. We write $\alpha(x)$ as the unique $\alpha$-limit set of $O^{-}(x)$ if $x$ admits a unique negative semiorbit $O^{-}(x)$. If $O^-(x)$ is precompact, then $\alpha(x)$ is nonempty, compact, connected, and invariant.

A closed set $C\subset X$ is called a cone of rank-$k$ ({\it abbr. $k$-cone}) if \begin{itemize}
\item[{(\rm i)}] For any $v\in C$ and $l\in \RR,$ $lv\in C$; 
\item[{\rm ii)}]  $\max\{\dim W:C\supset W \text{ linear subspace}\}=k.$
\end{itemize}
\noindent Moreover, the integer $k(\ge 1)$ is called {\it the rank of $C$}. A $k$-cone $C\subset X$ is said to be {\it solid} if its interior
$\tint\, C\ne \emptyset$; and $C$ is called {\it $k$-solid} if there is a $k$-dimensional linear subspace $W$ such that
$W\setminus \{0\}\subset \tint C$. Given a $k$-cone $C\subset X$, we say that $C$ is {\it complemented} if there exists a $k$-codimensional subspace $H^{c}\subset X$ such that $H^{c}\cap C=\{0\}$. For two points $x,y\in X$, we call that {\it $x$ and $y$ are ordered}, denoted by $x\thicksim y$, if $x-y\in C$. Otherwise, $x,y$ are called to be {\it unordered}, denoted by $x\rightharpoondown y$. The pair of points $x,y\in X$ are said to be {\it strongly ordered}, denoted by $x\thickapprox y$, if $x-y\in \tint \,C$. A nonempty set $W\subset X$ is called ordered if $x\thicksim y$ for any $x,y\in W$ and it is called (resp. strongly ordered) unordered if it is not a singleton and (resp. $x\thickapprox y$) $x\rightharpoondown y$ for any two distinct points $x,y\in W$. Two sets $A,\,B\subset X$ are said to be {\it ordered}, denoted by $A\thicksim B$, if $x\thicksim y$ for any $x\in A$ and $y\in B$.

The {\it separation index} between set $L_1$ and $L_2$, denoted by $\underline{\text{dist}}(L_1,L_2)$, is defined as 

$$\underline{\text{dist}}(L_1,L_2)=\inf\limits_{v\in L_1, \norm{v}=1}\{\inf\limits_{u\in L_2}\norm{v-u}\}.$$ A linear operator $R\in L(X)$ is called {\it strongly positive with respect to $C$}, if $R\,\big(C\setminus\{0\}\big)\subset \text{Int} C$. It is called {\it strongly focusing with respect to $C$}, if there is a $\kappa>0$ such that $0\notin R(C\setminus \{0\})$ and

$$\underline{\text{dist}}(RC, X\setminus C)=\kappa,$$ where $\kappa$ is called {\it the separation index of $R$}. The strongly focusing operator $R$ is automatically a strongly positive operator. On $\mathbb{R}^n$, the strong positivity of $R$ indicates that $R$ possesses the strongly focusing property. Denote $d(x,y)=\norm{x-y}$ for any $x,y\in X$, and $d(x,B)=\inf\limits_{y\in B}d(x,y)$ for any $x\in X, B\subset X$.

A semiflow $\Phi_t$ on $X$ is called {\it monotone with respect to $C$} if $\Phi_t(x)\thicksim\Phi_t(y)\,\,\text{whenever}\, x\thicksim y\,\, \text{and}\,\,t\ge 0;$ and $\Phi_t$ is called {\it strongly monotone with respect
to $C$} if $\Phi_t$ is monotone with respect to $C$ and $\Phi_t(x)\approx \Phi_t(y)\,\,\text{whenever}\, x\ne y,\,x\thicksim y\,\text{and}\, t>0.$ A nontrivial positive semiorbit $O^+(x)$ is called {\it pseudo-ordered} (also called {\it of Type-I}) if there exist two distinct points
$\Phi_{t_1}(x),\Phi_{t_2}(x)$ in $O^+(x)$ such that $\Phi_{t_1}(x)\thicksim\Phi_{t_2}(x)$. Otherwise, $O^+(x)$ is called {\it unordered}
(also called {\it of Type-II}). Hereafter, we let 

$$Q=\{x\in X: O^+(x) \text{ is pseudo-ordered}\}.$$

Throughout this paper, we assume that $C$ is a complemented $k$-solid cone, and $\Phi_t$ with its compact $x$-derivative $D_x\Phi_t$ ($x\in X, t>0$) admits a flow extension on each $\omega$-limit set.

\begin{definition}\label{strongly focusing monotone} {\rm A semiflow $\Phi_t$ is called {\it strongly focusing monotone with respect to $C$}, if it satisfies:

(i) It is $C^1$-smooth and strongly monotone with respect to $C$ such that the $x$-derivative $D_x\Phi_t$ of $\Phi_t \,(t>0)$ is strongly positive with respect to $C$ for any $x\in X$;

(ii) For each compact invariant set $\Sigma$ with respect to $\Phi_t$, one can find constants $\delta,T,\kappa>0$ such that for any $z\in \Sigma$ and $x, y\in B_{\delta}(z)$, there is a strongly focusing operator $T_{(x,y)}$ with respect to $C$ such that its separation index is greater than $\kappa$, and $T_{(x,y)}(x-y)=\Phi_T(x)-\Phi_T(y)$, where $B_{\delta}(z)=\{v\in X:\norm{v-z}\leq\delta\}$.}
\end{definition}

\begin{rem} {\rm Let $\Phi_t$ be a $C^1$-smooth flow strongly monotone w.r.t. $C$ on $\mathbb{R}^n$, whose $x$-derivative $D_x\Phi_t$ is strongly positive w.r.t. $C$ for any $x\in\mathbb{R}^n$ and $t>0$. Then, $\Phi_t$ is strongly focusing monotone w.r.t. $C$.}
\end{rem}

\begin{rem} {\rm The restriction ``for each compact invariant set $\Sigma$'' in the strongly focusing monotonicity in Definition \ref{strongly focusing monotone}(ii) can be relaxed and becomes ``for each $\omega$-limit set $\omega(x)$'' in the proof of the results in this paper. Specially, in the proof of the pseudo-ordered principle (see Theorem A) and the solid Poincar\'{e}-Bendixson theorem (see Theorem C), the restriction in the strongly focusing monotonicity can be weakened as ``for the subset of the $\omega$-limit set, consisting of equilibiria''.}
\end{rem}

\begin{rem}{\rm Let $\tilde{\Sigma}=\text{Co}\{B_{\delta}(\Sigma)\}\times\text{Co}\{B_{\delta}(\Sigma)\}$, where $B_{\delta}(\Sigma)=\{v\in X:d(v,\Sigma)\leq\delta\}$ and $\text{Co}\{B_{\delta}(\Sigma)\}$ is the convex hull of $B_{\delta}(\Sigma)$. Let $T_{(x,y)}=\int_0^1D_{y+s(x-y)}\Phi_Tds$ for any $(x,y)\in\tilde{\Sigma}$. Then, one has $T_{(x,y)}(x-y)=\Phi_T(x)-\Phi_T(y)$. Let $\kappa>0$. Compared with Definition \ref{strongly focusing monotone}(ii), the following condition has more restriction.\\
($*$)\quad  $T_{(x,y)}$ for any $(x,y)\in \tilde{\Sigma}$ is strongly focusing w.r.t. $C$ such that its separation index is greater than $\kappa$.}
\end{rem}

Let $G(k,X)$ be {\it the Grassmanian of $k$-dimensional linear subspaces of $X$}, which consists of all $k$-dimensional linear subspaces in $X$. $G(k, X)$ is a completed metric space by endowing {\it the gap metric} (see, for example, \cite{Kato,LL}). More precisely, for any nontrivial
closed subspaces $L_1,L_2\subset X$, define that

\begin{equation*}\label{E:GapDistance}
 d(L_1,L_2)=\max\left\{\sup_{v\in L_1\cap S}\inf_{u\in L_2\cap S}\norm{v-u}, \sup_{v\in L_2\cap S}\inf_{u\in L_1\cap S}\norm{v-u}\right\},
 \end{equation*} where $S=\{v\in X:\norm{v}=1\}$ is the unit sphere.
For a $k$-solid cone $C\subset X$, we denote by $\G_k(C)$ the set of $k$-dimensional subspaces inside $C$, that is,

$$\G_k(C)=\{L\in G(k,X):\,L\subset C\}.$$
Let $\Sigma\subset X$ be a compact invariant subset of $\Phi_t$. We consider the linear skew-product semiflow $(\Phi_t,\,D\Phi_t)$ on $\Sigma\times X$,
which is defined as $(\Phi_t, D\Phi_t)(x,v)=(\Phi_t(x),D_x\Phi_t v)$ for any $(x,v)\in \Sigma\times X$ and $t>0$. Let $\{E_x\}_{x\in \Sigma}$ be a family of $k$-dimensional subspaces of $X$. We call $\Sigma\times (E_x)$ {\it a $k$-dimensional continuous vector bundle on $X$} if the map $\Sigma\mapsto G(k,\,X): x\mapsto E_x$ is continuous. Let $\{F_x\}_{x\in \Sigma}$ be a family of $k$-codimensional closed vector subspaces of $X$. We call $\Sigma\times (F_x)$ {\it a $k$-codimensional continuous vector bundle on $X$} if there is a $k$-dimensional continuous vector bundle $\Sigma\times (L_x)\subset \Sigma\times X^*$ such that the kernel ${\rm Ker}(L_x)=F_x$ for each $x\in \Sigma$.
Here, $X^*$ is the dual space of $X$.

Let $\Sigma\times (E_x)$ be a $k$-dimensional continuous vector bundle on $X$, and let $\Sigma\times (F_x)$ be a $k$-codimensional continuous vector bundle
on $X$ such that $X=E_x\oplus F_x$ for all  $x\in \Sigma$. We define the {\it family of projections associated with the decomposition}
$X=E_x\oplus F_x$ as $\{\Pi^{E_x}\}_{x\in \Sigma}$ where $\Pi^{E_x}$ is the linear projection of $X$ onto $E_x$ along $F_x$, for each $x\in \Sigma$. Write $\Pi^{F_x}=\text{I}-\Pi^{E_x}$ for each $x\in \Sigma$. Clearly, $\Pi^{F_x}$ is the linear projection of $X$ onto $F_x$ along $E_x$. Moreover, both $\Pi^{E_x}$ and $\Pi^{F_x}$ are continuous with respect to $x\in \Sigma$. We say that the decomposition $X=E_x\oplus F_x$ is
{\it invariant with respect to $(\Phi_t,\,D\Phi_t)$} if $D_x\Phi_tE_x=E_{\Phi_t(x)}$, $D_x\Phi_tF_{x}\subset F_{\Phi_t(x)}$ for each $x\in \Sigma$ and $t\ge 0$.

\begin{definition}\label{D:ES-separation}
{\rm Let $\Sigma\subset X$ be a compact invariant subset for $\Phi_t$. The linear skew-product semiflow $(\Phi_t,\,D\Phi_t)$ admits a
{\bf $k$-exponential separation along $\Sigma$} (for short, $k$-exponential separation), if there are $k$-dimensional continuous vector bundle
$\Sigma\times (E_x)$ and $k$-codimensional continuous vector bundle $\Sigma\times (F_x)$ such that \begin{itemize}

\item[{\rm (i)}] $X=E_x\oplus F_x$ for any $x\in \Sigma$;

\item[{\rm (ii)}] $D_x\Phi_tE_x=E_{\Phi_t(x)}$, $D_x\Phi_tF_{x}\subset F_{\Phi_t(x)}$ for any $x\in \Sigma$ and $t\geq 0$;

\item[{\rm (iii)}]  there are constants $M>0$ and $0<\gamma<1$ such that
\begin{equation}\label{dominated-splitting}\norm{D_x\Phi_tw}\leq M\gamma^{t}\norm{D_x\Phi_tv}\end{equation} for all $x\in \Sigma$, $w\in F_x\cap S$, $v\in E_{x}\cap S$
and $t\ge 0$, where $S=\{v\in X:\norm{v}=1\}$.\end{itemize}

Let $C\subset X$ be a complemented $k$-solid cone. If, in addition,

\begin{itemize}\item[{\rm (iv)}] $E_x\subset {\rm Int}\,C\cup \{0\}$ and $F_x\cap C=\{0\}$ for any $x\in \Sigma$,\end{itemize}

\noindent then $(\Phi_t,\,D\Phi_t)$ is said to admit a {\bf $k$-exponential separation along $\Sigma$ associated with $C$}.}
\end{definition}

Since $E_x$ is $k$ dimensional for any $x\in \Sigma$, one can define {\it the infimum norm $m(D_x\Phi_t|_{_{E_x}})$
of $D_x\Phi_t$ restricted on $E_x$} for each $x\in \Sigma$ and $t\geq 0$ as follows:
\begin{equation}\label{infimum norm}m(D_x\Phi_t|_{_{E_x}})=\inf\limits_{v\in E_x\cap S}\norm{D_x\Phi_tv},\end{equation} where $S=\{v\in X: \norm{v}=1\}$.

\begin{definition}\label{Lyapunov expnents}{\rm For each $x\in \Sigma$, {\it the $k$-Lyapunov exponent} is defined as
\begin{equation}
\lambda_{kx}=\limsup_{t\to +\infty}\dfrac{\log m(D_x\Phi_t|_{_{E_x}})}{t},
\end{equation}} where $\Sigma\subset X$ is a compact invariant subset of $\Phi_t$.
\end{definition} A point $x\in \Sigma$ is called {\it a regular point} if $\lambda_{kx}=\lim\limits_{t\to +\infty}\dfrac{\log m(D_x\Phi_t|_{_{E_x}})}{t}$.

\vskip 3mm
We now state our main results:

{\bf \noindent Theorem A}\label{Order-property} {\bf (Pseudo-Ordered Principle)} Assume that $\Phi_t$ is strongly focusing monotone with respect to the $k$-cone $C$. Let $O^+(x)$ be a precompact pseudo-ordered semiorbit. Then, $\omega(x)$ is ordered, and more, it is topologically conjugate to a compact invariant set of a continuous vector field in $\mathbb{R}^k$.

\begin{rem} {\rm We use the infinite dimensional version of pseudo-ordered principle to extend our previous work \cite[Theorem A]{F-W-W4} on the ``total-ordering property" of the $\omega$-limit set of a pseudo-ordered orbit for the case $X\subset \mathbb{R}^n$, which gives a positive answer of the open problem posed by Sanchez \cite[P.1984]{San09}.  For the class of semiflows strongly focusing monotone w.r.t. $k$-cones, we improve the order-trichotomy (see \cite[Theorem B]{F-W-W1}) of the $\omega$-limit set $\omega(x)$ of a pseudo-ordered semiorbit for continuous semiflows strongly monotone w.r.t. $k$-cones. It should be pointed out that $\omega(x)$ is actually contained in $\mathbb{R}^k$ due to Theorem A.}
\end{rem}

Let 

$$C_{E}=\{x\in X:\,\omega(x)\,\text{consists of a singleton}\}.$$ By virtue of Theorem A, we have the solid version of generic dynamics theorem.

{\bf \noindent Theorem B}\label{Generic Dynamics} {\bf (Generic Solid Dynamics Theorem)} Assume that $\Phi_t$ is a $C^{1,\alpha}$-smooth semiflow strongly focusing monotone with respect to the $k$-cone $C$. Let $\mathcal{D}\subset X$ be an open bounded set such that $O^+(\mathcal{D})$ is precompact. Then, $\text{Int}(Q\cup C_E)$ (interior in $X$) is dense in $\mathcal{D}$, and $\omega(x)$ is ordered for any $x\in \text{Int}(Q\cup C_E)$.

Based on Theorem A, we obtain the infinite dimensional solid version of Poincar\'{e}-Bendixson theorem for $2$-cones.

{\bf \noindent Theorem C}\label{poincare-bendixson-2} {\bf (Solid Poincar\'{e}-Bendixson Theorem)} Assume that $\Phi_t$ is strongly focusing monotone with respect to the $k$-cone $C$ with $k=2$. Let $O^+(x)$ be a precompact pseudo-ordered semiorbit. Then, the $\omega$-limit set $\omega(x)$ is either a nontrivial periodic orbit or a set consisting of equilibria with their potential connected orbits.

Together with Theorem B and C, we have the following generic solid Poincar\'{e}-Bendixson theorem.

{\bf \noindent Theorem D}\label{Generic Poincare} {\bf (Generic Solid Poincar\'{e}-Bendixson Theorem)} Assume that $\Phi_t$ is a $C^{1,\alpha}$-smooth semiflow strongly focusing monotone with respect to the $k$-cone $C$ with $k=2$. Let $\mathcal{D}\subset X$ be an open bounded set such that $O^+(\mathcal{D})$ is precompact. Then there exists an open dense subset $ \mathcal{D}^{'}\subset \mathcal{D}$ such that for any $x\in \mathcal{D}^{'}$, $\omega(x)$ is either a nontrivial periodic orbit or a set consisting of equilibria with their potential connected orbits.

\begin{rem} {\rm Consider the cone $C$ with its rank being $k=2$. Based on Theorem A, we improve the Poincar\'{e}-Bendixson theorem\cite[Theorem C]{F-W-W1}, which concludes that the $\omega$-limit set $\omega(x)$ of a pseudo-ordered semiorbit $O^+(x)$ containing no equilibrium is a periodic orbit. By Theorem C, we make a whole description on the ordering structure of $\omega(x)$, and extend \cite[Theorem B]{F-W-W4} to the case of infinite dimensional spaces. In Theorem D, we show that $\omega$-limit sets are generically contained in 2-dimensional manifolds.}
\end{rem}

\section{Preliminary Results}\label{S:PR}
Before proving our main theorems, we present in this section several lemmas, some of which were obtained in our previous studies \ucite{F,F-W-W1,F-W-W2,F-W-W3,F-W-W4}, and are refined here. The first lemma is on the order-trichotomy of the $\omega$-limit set of a pseudo-ordered semiorbit.

\begin{lemma}\label{ordered trichotomy} Let $\Phi_t$ be a semiflow strongly monotone with respect to $C$. If $x\in Q$ such that $O^+(x)$ is precompact, then the closure of any full-orbit in $\omega(x)$ is ordered. Moreover, one of the following alternatives must occur: either,

\begin{itemize}
\item[{\rm (a)}] $\omega(x)$ is ordered; or,
\item[{\rm (b)}] $\omega(x)\subset E$ is unordered; or otherwise,
\item[{\rm (c)}] there is an ordered invariant subset $\tilde{B}\subsetneq \omega(x)$ such that $\tilde{B}\thicksim \omega(x)$ and for any $z\in \omega(x)\setminus\tilde{B}$,

$$\alpha(z)\cup\omega(z)\subset \tilde{B} \,\,\text{ and }\,\,\alpha(z)\subset E.$$
\end{itemize}
\end{lemma}
\begin{proof}See \cite[Theorem A and B]{F-W-W1}.
\end{proof}

For the alternative (c), more information can be obtained as follows:
\begin{lemma}\label{T-R-ordered trichotomy} If the alternative {\rm (c)} holds, then

\begin{itemize}
\item[{\rm (i)}] $\tilde{B}\thicksim \omega(x)$;
 \item[{\rm (ii)}] there exist $z,\tilde{z}\in\omega(x)\setminus\tilde{B}$ such that $z \rightharpoondown \tilde{z}$;
\item[{\rm (iii)}] $\alpha(z)=\alpha(\tilde{z})\subset E$. Moreover, for any sequence $t_k\rightarrow\infty$ with $\Phi_{-t_k}(z)\to q$ and $\Phi_{-t_k}(\tilde{z})\to \tilde{q}$, one has $q=\tilde{q}\in E$. In particular,

\begin{equation}\label{E:z-z-asym}
\norm{\Phi_{-t}(z)-\Phi_{-t}(\tilde{z})}\to 0, \text{ as } t\to \infty.
\end{equation}
Here, $\Phi_{-t}$ is the inverse of $\Phi_t$ on $\omega(x)$ for any $t\geq0$.
\end{itemize}
\end{lemma}

\begin{proof} See \cite[Lemma 3.2]{F-W-W4} (which holds on the general Banach space $X$).
\end{proof}

\begin{lemma}\label{invariant subspace}
Let $\Phi_t$ be a semiflow strongly focusing monotone with respect to $C$. Then, $(\Phi_t,\,D\Phi_t)$ admits a $k$-exponential separation along any nonempty compact subset $\Sigma\subset E$, with $k$-dimensional continuous invariant vector bundle $\Sigma\times (E_p)$ and $k$-codimensional continuous invariant vector bundle $\Sigma\times (F_p)$ such that, for any $p\in \Sigma$,
\begin{itemize}
\item [{\rm (i)}] $E_p \setminus\{0\}\subset{\rm Int}\,C$, $F_p\cap C=\{0\}$ and $E_p\bigoplus F_p=X$,
\item [{\rm (ii)}] $D_p\Phi_t (E_p)=E_p$ and $D_p\Phi_t(F_p)\subset F_p$, for any $t\geq 0$ {\rm(}and hence, $(D_p\Phi_{t}\mid_{E_p})^{-1}$ exists{\rm)},
\item [{\rm (iii)}] there exist constants $M_{1,p},M_{2,p}>1$ and $\alpha_p, \beta_p\in \mathbb{R}$ with $\alpha_p<\beta_p$ such that

\begin{equation}\label{E:ab-control}
\norm{D_p\Phi_t\mid_{F_p}}< M_{1,p} e^{\alpha_p t}\,\,\text{ and }\,\,
\norm{(D_p\Phi_{t}\mid_{E_p})^{-1}}< M_{2,p} e^{-\beta_p t}
\end{equation} for any $t\geq 0$. Moreover, $\beta_p,\alpha_p,M_{1,p},M_{2,p}$ can be taken as local constant functions on $\Sigma$.
\end{itemize}
\end{lemma}

\begin{proof} For any given compact set $\Sigma\subset E$, the strong positivity and compactness of $D_p\Phi_t$ for any $p\in X$ and $t>0$ implies that the linear skew-product semiflow $(\Phi_t,\,D\Phi_t)$ admits a $k$-exponential separation along $\Sigma$ associated with $C$ (see, e.g. \cite[Theorem 4.1 and Corollary 2.2]{Tere}). So, there exist a $k$-dimensional continuous invariant vector bundle $\Sigma\times (E_p)$ and a $k$-codimensional continuous invariant vector bundle $\Sigma\times (F_p)$, which satisfies (i)-(ii) for any $p\in \Sigma$.

It remains to prove (iii). By virtue of \eqref{dominated-splitting}, one can find $M>0$ and $0<\gamma<1$ such that

\begin{equation}\label{E:EX-norm}
\norm{D_p \Phi_t\mid_{F_p}}\leq M\gamma^{t}\norm{D_p\Phi_t v}\,\text{ for all }t\ge 0 \,\,\text{and}\,\, v\in E_p\cap S,
\end{equation} where $S=\{v\in X:\norm{v}=1\}$. Denoted by $\sigma(D_p\Phi_t)$, $\sigma(D_p\Phi_t|_{E_p})$ and $\sigma(D_p\Phi_t|_{F_p})$ the spectrum of $D_p\Phi_t$ on $X$, $D_p\Phi_t|_{E_p}$ on $E_p$, and $D_p\Phi_t|_{F_p}$ on $F_p$ for any $p\in\Sigma$ and $t>0$ respectively. Note that $E_p$ is $k$ dimensional. One has that $\sigma(D_p\Phi_t|_{E_p})$ consists of eigenvalues with eigenvectors contained in the complexification of $E_p$ (which is also denoted by $E_p$). For given $t_0>0$, $\sigma(D_p\Phi_{t_0}\mid_{E_p})=\{\lambda_1,\lambda_2,\cdots,\lambda_k\}$ (counted with multiplicity) such that $\abs{\lambda_1}\geq\abs{\lambda_2}\geq\cdots\abs{\lambda_k}$. By virtue of the spectral mapping theorem (cf. \cite[Theorem 2.52]{Chi}), there is a constant $\tilde{\beta}_{p,\lambda_j}$ such that $e^{\tilde{\beta}_{p,\lambda_j}t_0}=\lambda_{j}$ for each $j\in\{1,2,\cdots,k\}$; and more, $\sigma(D_p\Phi_t\mid_{E_p})=\{e^{\tilde{\beta}_{p,\lambda_1}t},e^{\tilde{\beta}_{p,\lambda_2}t},\cdots,e^{\tilde{\beta}_{p,\lambda_k}t}\}$ for each $t>0$. Clearly, $\text{Re}(\tilde{\beta}_{p,\lambda_1})\geq \text{Re}(\tilde{\beta}_{p,\lambda_2})\geq\cdots \text{Re}(\tilde{\beta}_{p,\lambda_k})$. For each $\lambda_j$, there is an eigenvector $u_j+iv_j$ of $D_p\Phi_{t_0}$ such that $D_p\Phi_{t_0}(u_j+iv_j)=\lambda_j(u_j+iv_j)$, where $i$ is the imaginary number. It implies that $D_p\Phi_{nt_0}(u_j+iv_j)=(\lambda_j)^n(u_j+iv_j)$, $D_p\Phi_{nt_0}(u_j-iv_j)=(\overline{\lambda_j})^n(u_j-iv_j)$ for any $j\in\{1,2,\cdots,k\}$ and $n\in \mathbb{N}^+$. One has that 

$$\begin{aligned}|\lambda_j|^n(\norm{u_j+iv_j}+\norm{u_j-iv_j})&=\norm{D_p\Phi_{nt_0}|_{E_p}(u_j+iv_j)}+\norm{D_p\Phi_{nt_0}|_{E_p}(u_j-iv_j)}\\
&\geq \norm{D_p\Phi_{nt_0}|_{E_p}u_j}+\norm{D_p\Phi_{nt_0}|_{E_p}v_j}\\&\overset{\eqref{E:EX-norm}}{\geq} M^{-1} \gamma^{-nt_0}(\norm{u_j}+\norm{v_j})\norm{D_p\Phi_{nt_0}\mid_{F_p}}
\end{aligned}$$ for any $n\in \mathbb{N}^+$. It then follows that $|\lambda_j|^n \gamma^{nt_0}\frac{M(\norm{u_j+iv_j}+\norm{u_j-iv_j})}{\norm{u_j}+\norm{v_j}}\geq \norm{D_p\Phi_{nt_0}\mid_{F_p}}$ for any $n\in \mathbb{N}^+$. Together with Gelfand's formulation on the spectral radius $r(D_p\Phi_{t_0}\mid_{F_p})$ of bounded operator $D_p\Phi_{t_0}\mid_{F_p}$, that is, $r(D_p\Phi_{t_0}\mid_{F_p})=\lim\limits_{n\rightarrow +\infty}\norm{D_p\Phi_{nt_0}\mid_{F_p}}^{\frac{1}{n}}$ (see \cite[Theorem 1.2.7]{Mur}), one has that 

$$\gamma\abs{\lambda_j}\geq r(D_p\Phi_{t_0}\mid_{F_p})$$ for each $j\in\{1,2,\cdots,k\}$. Let $\tilde{\beta}_p=\text{Re}(\tilde{\beta}_{p,\lambda_k})$ and take $\tilde{\alpha}_p\in \mathbb{R}$ such that $e^{\tilde{\alpha}_p t_0}=r(D_p\Phi_{t_0}\mid_{F_p})$. Recall that $\gamma\in (0,1)$. One has that $\tilde{\alpha}_p<\tilde{\beta}_p$. Then, we have

$$\lim\limits_{n\rightarrow \infty}\norm{(D_p\Phi_{n t_0}|_{F_p})}^{\frac{1}{n}}=e^{\tilde{\alpha}_pt_0}\,\,\text{ and }\,\, \lim\limits_{n\rightarrow \infty}\norm{\big[(D_p\Phi_{t_0}|_{E_p})^{-1}\big]^{n}}^{\frac{1}{n}}=e^{-\tilde{\beta}_pt_0}.$$

Take $\beta_p=\tilde{\beta}_p-\frac{\tilde{\beta}_p-\tilde{\alpha}_p}{4}$ and $\alpha_p=\tilde{\alpha}_p+\frac{\tilde{\beta}_p-\tilde{\alpha}_p}{4}$. By repeating arguments in \cite[the part of the proof of Lemma 3.5 on page 9]{F-W-W4}, we obtain that for any $p\in\Sigma$, there are constants $M_{1,p}$, $M_{2,p}>1$ and a small open neighborhood $\mathcal{V}_p$ of $p\in\Sigma$ such that 

$$\norm{D_q\Phi_{t}|_{F_q}}<M_{1,p}e^{\alpha_pt}\,\,\text{ and }\,\,\norm{(D_q\Phi_{t}|_{E_q})^{-1}}<M_{2,p}e^{-\beta_pt}$$ for any $t\geq0$ and $q\in \mathcal{V}_p\cap\Sigma$. Thus, we can take $\beta_p,\alpha_p, M_{1,p}, M_{2,p}$ as local constant functions w.r.t. $p\in \Sigma$. Thus, we have proved (iii).

Therefore, We have completed the proof of this Lemma.
 \end{proof}

Recall that $\Pi^{F_p}$ are the natural projection on $F_p$ along $E_p$ and $\Pi^{E_p}=\text{I}-\Pi^{F_p}$, where $E_p$ and $F_p$ are the invariant subspaces mentioned in Lemma \ref{invariant subspace}. We have the following lemma.

\begin{lemma}\label{decompostion ratio} Let $\Sigma\subset E$ be a nonempty compact set. Then,
\item{(i)} the projections $\Pi^{E_p}$ and $\Pi^{F_p}$ are bounded uniformly for $p\in \Sigma$.
\item{(ii)} there exists $\kappa_{\Sigma}>0$ such that, if $v\in X\setminus\{0\}$ satisfies $\norm{\Pi^{E_p}v}\geq \kappa_{\Sigma}\norm{\Pi^{F_p}v}$ for some $p\in\Sigma$, then $v\in {\rm Int}\,C$.
\end{lemma}

\begin{proof} See \cite[Lemma 2.9]{F}.
\end{proof}

\vskip 5mm
Let  $\mathcal{M}_x=\{y\in X:\,y\thicksim x\,\text{and}\,y\neq x\}$ for any $x\in X$.

\begin{lemma}\label{L:instabilty}Let $\Phi_t$ be a semiflow strongly focusing monotone with respect to $C$. Assume that $\lambda_{kz}>0$ for any $z$ contained in the nonempty comapct  set $\omega(x)$. Then, there exists a constant $\delta^{''}>0$ such that for any $y\in \mathcal{M}_x$

$$\limsup_{t\to +\infty}\norm{\Phi_t(y)-\Phi_t(x)}\ge \delta^{''}.$$
\end{lemma}

\begin{proof} See \cite[Lemma 4.2]{F}.\end{proof}
 
\begin{lemma}\label{p-notin-omega.y}
Let $\Phi_t$ be a semiflow strongly focusing monotone with respect to $C$. Let $p\in E$ and $t_0>0$. Assume that $\abs{\lambda}>1$ for any $\lambda\in \sigma(D_p\Phi_{t_0}\mid_{E_p})$, where $\sigma(D_p\Phi_{t_0}\mid_{E_p})$ is the spectrum of $D_p\Phi_{t_0}\mid_{E_p}$. Then, there exist constants $\rho, \kappa^*, T>0$ and $\tilde{T}$ with $\tilde{T}>2T$ such that if a point $y\in \mathcal{M}_p$ satisfies $\Phi_t(y)\in B_{\rho}(p)$ for any $t\in [0, N\tilde{T}]$ with $N\in \mathbb{N}$, then for any non-negative integer $n\leq N$, ones have 

\begin{equation}\label{Lower-bound-growth-estimates}\begin{aligned}
&\quad \norm{\Phi_{n\tilde{T}+\tau}(y)-p}\geq 3^{n}\kappa^*\norm{\Phi_T(y)-p}\,\, \text{with}\,\,\tau\in[T,\tilde{T}+T],\\
&\text{and}\\
&\quad \norm{\Phi_{n\tilde{T}+\tau}(y)-p}\geq 3^{n+1}\norm{\Phi_T(y)-p}\,\, \text{with}\,\,\tau\in[\tilde{T},\tilde{T}+T].
\end{aligned}\end{equation} Consequently,  if $N>0$, then

\begin{equation}\label{Lower-bound-growth-estimates-coro} \norm{\Phi_{n\tilde{T}+\tau}(y)-p}\geq 3^{n}\norm{\Phi_T(y)-p}\,\, \text{with}\,\,\tau\in[0,T].\end{equation} for any positive integer $n\leq N+1$.
\end{lemma}

\begin{proof} By the strongly focusing monotonicity of $\Phi_t$, there are constants $\delta,T,\kappa>0$ such that there exists a strongly focusing operator $T_{(y,x)}$ such that its separation index is greater than $\kappa$, and $T_{(y,x)}(y-x)=\Phi_T(y)-\Phi_T(x)$ for any $y, x\in B_{\delta}(p)$. Hence, 

\begin{equation}\label{separation around equilibrium}d(\frac{\Phi_T(y)-p}{\norm{\Phi_T(y)-p}}, X\setminus C)\geq \kappa\end{equation} for any $y\in \mathcal{M}_p\cap B_{\delta}(p)$. Clearly, $\kappa\leq 1.$ By virtue of Lemma \ref{invariant subspace}, there exist invariant subspaces $E_p$ and $F_p$ such that (\ref{E:ab-control}) holds for any $t\geq 0$. By Lemma \ref{decompostion ratio}, $\Pi^{E_p}$ is bounded. It then follows from (\ref{separation around equilibrium}) that 

\begin{equation}\label{Mainsubspace-estimates}\kappa\cdot\norm{\Phi_T(y)-p}\leq\norm{\Pi^{E_p}(\Phi_T(y)-p)}\leq\norm{\Pi^{E_p}}\cdot
\norm{\Phi_T(y)-p}\end{equation} for any $y\in \mathcal{M}_p\cap B_{\delta}(p)$. Since $\abs{\lambda}>1$ for any $\lambda\in\sigma(D_p\Phi_{t_0}\mid_{E_p})$, the constant $\beta_p$ in (\ref{E:ab-control}) is greater than $0$. Then, one can take $\tilde{T}>2T$ such that 

\begin{equation}\label{tilde{T}-1}\frac{\kappa\cdot(e^{\beta_p t}-\frac{1}{2})}{M_{2,p}\norm{\Pi^{E_p}}}>3\end{equation} for any $t\geq\tilde{T}-T$, where $M_{2,p}$ is the one of constants in (\ref{E:ab-control}). By the smoothness of $\Phi_t$, there exists a $\tilde{\rho}\in (0,\delta)$ and a $\rho\in(0, \tilde{\rho})$ such that 

\begin{equation}\label{C1-Lip-constant}\begin{aligned} &\norm{D_u\Phi_t-D_v\Phi_t}\leq  \frac{\kappa}{2M_{2,p}\norm{\Pi^{E_p}}},\\
&{\rm and}\quad \Phi_{T}(y)\in B_{\tilde{\rho}}(p)\end{aligned}\end{equation} for any $u,v\in B_{\tilde{\rho}}(p)$, $y\in\mathcal{M}_p\cap B_{\rho}(p)$ and $t\in[0,\tilde{T}+T]$. Thus, one has
 
\begin{equation}\label{prime-growth-estimates}\begin{aligned} &\norm{\Pi^{E_p}(\Phi_{t-T}\circ\Phi_T(y)-p)}\geq\norm{D_p\Phi_{t-T}\circ\Pi^{E_p}(\Phi_T(y)-p)}\\
&\quad\quad\quad\quad\quad\quad\quad\quad\quad\quad-\norm{\int_0^{1}\Pi^{E_p}\circ(D_{p+s(\Phi_T(y)-p)}\Phi_{t-T}-D_p\Phi_{t-T})(\Phi_T(y)-p)ds}\\
&\overset{(\ref{E:ab-control})+(\ref{C1-Lip-constant})}{\geq} M_{2,p}^{-1}e^{\beta_p\cdot (t-T)}\norm{\Pi^{E_p}(\Phi_T(y)-p)}-\frac{\kappa}{2M_{2,p}\norm{\Pi^{E_p}}}\cdot\norm{\Pi^{E_p}}\cdot\norm{\Phi_T(y)-p}\\
&\overset{(\ref{Mainsubspace-estimates})}{\geq}\big[\frac{\kappa\cdot e^{\beta_p\cdot(t-T)}}{M_{2,p}}
-\frac{\kappa}{2M_{2,p}}\big]\cdot\norm{\Phi_T(y)-p}
\end{aligned}\end{equation} for any $t\in [T,\tilde{T}+T]$ and $y\in \mathcal{M}_p\cap B_{\rho}(p)$. It then follows from (\ref{prime-growth-estimates}) that

\begin{equation}\label{Lower-bound-growth-estimates-Ep}\norm{\Pi^{E_p}(\Phi_t(y)-p)}\geq \kappa^*\cdot\norm{\Pi^{E_p}}\cdot\norm{\Phi_{T}(y)-p}
\end{equation} and

\begin{equation}\label{Lower-bound-growth-estimates-mid}\norm{\Phi_t(y)-p}\geq \frac{\norm{\Pi^{E_p}(\Phi_t(y)-p)}}{\norm{\Pi^{E_p}}} \geq \kappa^*\cdot\norm{\Phi_{T}(y)-p}
\end{equation} for any $t\in[T,\tilde{T}+T]$ and $y\in \mathcal{M}_p\cap B_{\rho}(p)$, where 

$$\kappa^{*}=\frac{\kappa}{2M_{2,p}\norm{\Pi^{E_p}}}.$$ Together with with (\ref{tilde{T}-1}) and (\ref{prime-growth-estimates}), one has \begin{equation}\label{edge-point-estimate-mid}\norm{\Phi_{\tilde{T}+\tau}(y)-p}\geq 3\cdot\norm{\Phi_{T}(y)-p} \end{equation} for any $\tau\in [0,T]$ and $y\in \mathcal{M}_p\cap B_{\rho}(p)$.

We now prove $\rho, \kappa^*, T>0$ and $\tilde{T}>2T$ are the desired constants. Let $y\in \mathcal{M}_p\cap B_{\rho}(p)$ such that $\Phi_t(y)\in B_{\rho}(p)$ for any $t\in[0,N\tilde{T}]$ with a $N\in\mathbb{N}^+$. Clearly, $\Phi_t(y)\in \mathcal{M}_p$ for any $t\in[0,N\tilde{T}]$. It then follows from (\ref{Lower-bound-growth-estimates-mid})-(\ref{edge-point-estimate-mid}) that for any non-negative integer $n\leq N$,  

\begin{equation}\label{M-I-esti}\begin{aligned}&\quad \norm{\Phi_{t+n\tilde{T}}(y)-p}\geq  \kappa^*\cdot\norm{\Phi_{T+n\tilde{T}}(y)-p}\,\,\text{with}\,\,t\in[T,\tilde{T}+T]\\
&\text{and}\\
&\quad \norm{\Phi_{t+n\tilde{T}}(y)-p}\geq 3\cdot\norm{\Phi_{T+n\tilde{T}}(y)-p}\,\,\text{with}\,\,t\in[\tilde{T},\tilde{T}+T].\end{aligned}
\end{equation}  By the mathematical induction, one has that for any non-negative integer $n\leq N$, 
\begin{equation}\label{right-part's-interval-esti}  \norm{\Phi_{n\tilde{T}+T}(y)-p}\geq 3^n\cdot\norm{\Phi_{T}(y)-p}
\end{equation} By (\ref{M-I-esti})-(\ref{right-part's-interval-esti}), one has

$$\begin{aligned}
&\quad \norm{\Phi_{n\tilde{T}+\tau}(y)-p}\geq 3^{n}\kappa^*\norm{\Phi_T(y)-p}\,\, \text{with}\,\,\tau\in[T, \tilde{T}+T]\\
& \text{and}\\
&\quad \norm{\Phi_{n\tilde{T}+\tau}(y)-p}\geq 3^{n+1}\norm{\Phi_T(y)-p}\,\, \text{with}\,\,\tau\in[\tilde{T}, \tilde{T}+T],\end{aligned}$$ for any non-negative integer $n\leq N$. As a consequence, 

$$\norm{\Phi_{n\tilde{T}+\tau}(y)-p}\geq 3^{n}\norm{\Phi_T(y)-p}\,\,\text{with}\,\,\tau\in[0,T],$$ for any positive integer $n\leq N+1$.

Therefore, we have completed the proof of this lemma.
\end{proof}

\section{Proofs of Main results}
Due to the strongly focusing monotonicity of $\Phi_t$, we can make the further analysis on alternatives (b) and (c) in Lemma \ref{ordered trichotomy}.
The following are two crucial lemmas for proving Theorem A.

\begin{lemma}\label{Lemma for ii} The alternative {\rm (b)} does not occur.\end{lemma}

\begin{proof}
Suppose that the alternative (b) holds. For any $q\in\omega(x)\subset E$, let $E_q$, $F_q$ be invariant subspaces mentioned in Lemma \ref{invariant subspace}. Since $\omega(x)$ is unordered, it is not a singleton. Then, the connectedness of $\omega(x)$ implies that there exists a sequence $\{q_n\}\subset \omega(x)\subset E$ such that $q_n\neq q$ and $q_n\rightarrow q$ as $n\rightarrow \infty$. Since $\Phi_t$ admits a flow extension on $\omega(x)$, $\Phi_t$ is a homomorphism on $\omega(x)$ for any $t>0$. For simplicity, we denote $\Phi_{-t}$ the inverse of $\Phi_t$ on $\omega(x)$ for any $t>0$. Moreover, for any $t>0$, one has 

\begin{equation}\begin{aligned}\label{eq:1}
\norm{\Pi^{F_q}(q_n-q)}&=\norm{\Pi^{F_q}\big[\Phi_t\circ\Phi_{-t}(q_n)-\Phi_t\circ\Phi_{-t}(q)\big]}\\
&\leq\norm{D_q\Phi_{t}\circ\Pi^{F_q}(\Phi_{-t}(q_n)-\Phi_{-t}(q))}\\
&\,\,\,\,\,\,+\norm{\int_{0}^{1}\Pi^{F_q}\circ[D_{q+s(q_n-q)}\Phi_{t}-D_q\Phi_{t}](\Phi_{-t}(q_n)-\Phi_{-t}(q))ds}\\
&\leq\norm{D_q\Phi_{t}\circ\Pi^{F_q}(q_n-q)}+\norm{\int_{0}^{1}\Pi^{F_q}\circ[D_{q+s(q_n-q)}\Phi_{t}-D_q\Phi_{t}](q_n-q)ds}.
\end{aligned}
\end{equation}
\noindent By utilizing the unorderedness of $\omega(x)\subset E$ again, we have $q_{n}\rightharpoondown q$ for all $n$. It then follows from Lemma \ref{decompostion ratio}(ii) that 

\begin{equation}\begin{aligned}\label{eq:2}\norm{\Pi^{F_q}(q_n-q)}>\frac{1}{1+\kappa_{\omega(x)}}\norm{q_n-q}\end{aligned}
\end{equation} for all $n$. Moreover, by (\ref{E:ab-control}), there exist $M_{1,q}, M_{2,q}>1$ and $\beta_q>\alpha_q$ such that

\begin{equation}\label{eq:3}\begin{aligned}
&\quad \norm{(D_q\Phi_{t}|_{E_q})^{-1}\circ \Pi^{E_q}(q_n-q)}\leq M_{2,q} e^{-\beta_q t}\norm{\Pi^{E_q}(q_n-q)}\\
&\text{and}\\
&\quad \norm{D_q\Phi_{t}\circ \Pi^{F_q}(q_n-q)}\leq M_{1,q} e^{\alpha_qt}\norm{\Pi^{F_q}(q_n-q)}\end{aligned}\end{equation} for any $t>0$ and $n\in \mathbb{N}$.

Now, we claim that $\alpha_q\geq 0$. Supposing otherwise, one can take $\tilde{T}>0$ such that

\begin{equation}\label{eq:4} M_{1,q}e^{\alpha_q \tilde{T}}<\frac{1}{4} \end{equation} It then follows from $(\ref{eq:1})$ and (\ref{eq:3})-(\ref{eq:4}) that

\begin{equation}\label{eq:5} \norm{\Pi^{F_q}(q_n-q)}
\leq\frac{\norm{\Pi^{F_q}(q_n-q)}}{4}+\norm{\int_{0}^{1}\Pi^{F_q}\circ[D_{q+s(q_n-q)}\Phi_{\tilde{T}}-D_q\Phi_{\tilde{T}}](q_n-q)ds}
\end{equation} for all $n\in \mathbb{N}^+$. Recall that $\Phi_t$ is $C^1$-smooth. One has $\norm{D_{u}\Phi_{\tilde{T}}-D_{v}\Phi_{\tilde{T}}}\rightarrow 0$ whenever $\norm{u-q}\rightarrow 0$ and $\norm{v-q}\rightarrow 0$. Together with (\ref{eq:2}) and (\ref{eq:5}), one has $\norm{\Pi^{F_q}(q_n-q)}<\frac{1}{2}\cdot\norm{\Pi^{F_q}(q_n-q)}$ for sufficiently large $n$. It is a contradiction. Thus, we have proved $\alpha_q\geq 0$.

By the claim above, we have $\beta_q>\alpha_q\geq 0$ for any $q\in\omega(x)\subset E$. It yields that the $k$-Lyapunov exponent $\lambda_{kq}$ for any $q\in\omega(x)\subset E$, possesses 

$$\lambda_{kq}=\lim\sup\limits_{t\rightarrow +\infty}\frac{\log(m(D_q\Phi_t|_{E_q}))}{t}\geq\beta_q>0.$$ Since $O^{+}(x)$ is pseudo-ordered, one can choose a $\tilde{\tilde{T}}>0$ such that $x\thickapprox\Phi_{\tilde{\tilde{T}}}(x)$. It then follows from Lemma \ref{L:instabilty} that there is a $\delta^{\prime\prime}>0$ such that

$$\lim\sup\limits_{t\rightarrow +\infty}\norm{\Phi_{t+\tilde{\tilde{T}}}(x)-\Phi_t(x)}\geq\delta^{\prime\prime}.$$ Hence, there is a sequence $\{t_n\}\subset\mathbb{R}^+$ such that $\lim\limits_{n\rightarrow +\infty}t_n=+\infty$ and $\norm{\Phi_{t_n+\tilde{\tilde{T}}}(x)-\Phi_{t_n}(x)}\geq\frac{\delta^{\prime\prime}}{2}$. By choosing subsequence, one has $\lim\limits_{n\rightarrow +\infty}\Phi_{t_n}(x)=p\in \omega(x)$. It entails that $\norm{\Phi_{\tilde{\tilde{T}}}(p)-p}\geq\frac{\delta^{\prime\prime}}{2}$, a contradiction to $\omega(x)\subset E$.

Therefore, we have completed the proof.
\end{proof}

\begin{lemma}\label{Lemma for iii} The alternative {\rm (c)} can not occur.
\end{lemma}

\begin{proof} Suppose that the alternative (c) holds. Then, by Lemma \ref{T-R-ordered trichotomy}(ii), there are \begin{equation}\label{eq:star}z,\,\tilde{z}\in\omega(x)\setminus\tilde{B}\,\,\text{such that}\,\,z\rightharpoondown\tilde{z},\end{equation} and Lemma \ref{T-R-ordered trichotomy}(iii) holds. Fix such $z\in\omega(x)\setminus \tilde{B}$. Clearly, $\alpha(z)\subset E$ is compact and invariant w.r.t. $\Phi_t$. It follows from Lemma \ref{decompostion ratio}(ii) that for any $q\in \alpha(z)$, there is a constant $\kappa_{\alpha(z)}>0$ such that there is $\kappa_{\alpha(z)}>0$ such that 

\begin{equation}\label{eq:M-2}\norm{\Pi^{E_q}u}<\kappa_{\alpha(z)} \norm{\Pi^{F_q}u}\,\,\text{for all} \,\,u\in X\setminus C\,\,\text{and}\,\, q\in \alpha(z).\end{equation} Hence, 
 
\begin{equation}\label{F-whole-norm} \frac{1}{1+\kappa_{\alpha(z)}}\cdot\norm{u}<\norm{\Pi^{F_q}u}\,\,\text{for all} \,\,u\in X\setminus C\,\,\text{and}\,\, q\in \alpha(z).\end{equation} By (\ref{E:ab-control}) in Lemma \ref{invariant subspace}, for any $q\in \alpha(z)$, there exist contants $M_{1,q}, M_{2,q}>1$ and $\alpha_q,\beta_q$ with $\alpha_q<\beta_q$ such that

\begin{equation}\label{eq:M-3}\begin{aligned}&\norm{D_q\Phi_{t}w}\geq \frac{e^{\beta_q t}}{M_{2,q}}\norm{w}\\
&\norm{D_q\Phi_{t}v}\leq M_{1,q}e^{\alpha_q t}\norm{v}
\end{aligned}\end{equation}for any $w\in E_q\setminus\{0\}$, $v\in F_q\setminus\{0\}$ and $t\geq 0$.

{\it We assert that there exists a $q\in\alpha(z)$ such that $\beta_{q}>0$.} Supposing otherwise, $\beta_q\leq 0$ for any $q\in \alpha(z)$. Then, $\alpha_q<\beta_q\leq0$ for any $q\in\alpha(z)$. By Lemma \ref{decompostion ratio}(i), 

$$M_3=\max\limits_{q\in\alpha(z)}\{\norm{\Pi^{F_q}}\}$$ is well-defined. By virtue of (\ref{eq:M-3}), for any $q\in\alpha(z)$, there exists $\nu(q)>0$ such that 

\begin{equation}\label{eq:M-5}\norm{D_q\Phi_{\nu(q)}v}\leq\frac{1}{4M_3(1+\kappa_{\alpha(z)})}\cdot\norm{v}\end{equation} when $v\in F_q\setminus\{0\}$. Notice that $M_{1,q}, M_{2,q},\alpha_{q},\beta_{q}$ are local constant functions on $\alpha(z)$ in (\ref{E:ab-control}), and $\alpha(p)\times (E_q)$, $\alpha(p)\times (F_q)$ are continuous vector bundles. One can take $\nu$ being a local constant function on $\alpha(z)$. Moreover, the compactness of $\alpha(z)$ yields that $\nu$ is bounded and $M=\max\limits_{q\in\alpha(z)}\{\nu(q)\}$ is well-defined. By (\ref{E:z-z-asym}) in Lemma \ref{T-R-ordered trichotomy}(iii), we have

\begin{equation}\label{Negative-conve}\norm{\Phi_{-t}(z)-\Phi_{-t}(\tilde{z})}\rightarrow 0\,\,\text{as}\,\, t\rightarrow \infty,\end{equation} where $\tilde{z}\in\omega(x)\setminus\tilde{B}$ is the point mentioned in (\ref{eq:star}). It then follows from the $C^1$-smoothness of $\Phi_t$ and the compactness of $\alpha(z)$ that there exists $\tilde{\delta}>0$ such that for $\tilde{q} \in\omega(x)$ and $q\in\alpha(z)$ satisfying $\norm{\tilde{q}-q}<2\tilde{\delta}$, one has

 \begin{equation}\label{eq:M-6}\norm{D_{\tilde{q}}\Phi_{\nu(q)}-D_{q}\Phi_{\nu(q)}}<\frac{1}{2M_3(1+\kappa_{\alpha(z)})}.\end{equation} By (\ref{Negative-conve}), there is a time $T_{\tilde{\delta}}>0$ such that $d(\Phi_{-t}(z),\alpha(z))<\frac{\tilde{\delta}}{2}$ and $d(\Phi_{-t}(\tilde{z}),\Phi_{-t}(z))<\tilde{\delta}$ for any $t\geq T_{\tilde{\delta}}$. Furthermore, one can find a curve $\{q_{-t}\}_{t\geq 0}\subset \alpha(z)$ such that $\norm{\Phi_{-t}(z)-q_{-t}}\leq2\cdot d(\Phi_{-t}(z),\alpha(z))\rightarrow 0$ as $t\rightarrow \infty$. Let $\tau_1=T_{\tilde{\delta}}$. By the continuity of $\nu(\cdot)$, there is a $\tau_{2}>0$ such that $\tau_{2}=\tau_{1}+\nu(q_{-\tau_{2}})$ and then, by mathematical induction, there is a point $\tau_{n+1}>\tau_n$ such that $\tau_{n+1}=\tau_{n}+\nu(q_{-\tau_{n+1}})$ for any $n\geq 1$. For simplicity, we write $b_n=\Phi_{-\tau_n}(\tilde{z})$, $c_n=\Phi_{-\tau_n}(z)$, and let $q_n=q_{-\tau_n}$ for any $n$. Clearly, $b_n\rightharpoondown c_n$ and $\norm{b_n-c_n}<\tilde{\delta}$ for all $n$. One has that

\begin{equation}\begin{aligned}\norm{\Pi^{F_{q_{n+1}}}(b_{n}-c_{n})}&\leq\norm{D_{q_{n+1}}\Phi_{\nu(q_{n+1})}\circ\Pi^{F_{q_{n+1}}}(b_{n+1}-c_{n+1})}\\
&\,\,\,\,+\norm{\int_{0}^{1}\Pi^{F_{q_{n+1}}}\circ[D_{c_{n+1}+s(b_{n+1}-c_{n+1})}\Phi_{\nu(q_{n+1})}-D_{q_{n+1}}\Phi_{\nu(q_{n+1})}](b_{n+1}-c_{n+1})ds}\\
&\overset{(\ref{eq:M-5})+(\ref{eq:M-6})}{\leq}\frac{1}{(1+\kappa_{\alpha(z)})}\cdot\big[\frac{\norm{\Pi^{F_{q_{n+1}}}(b_{n+1}-c_{n+1})}}{4M_3}+\frac{\norm{b_{n+1}-c_{n+1}}}{2}\big]\\
&\leq\frac{1}{(1+\kappa_{\alpha(z)})}\cdot\big[ \frac{\norm{\Pi^{F_{q_{n+1}}}}\cdot\norm{b_{n+1}-c_{n+1}}}{4M_3}+\frac{\norm{b_{n+1}-c_{n+1}}}{2}\big]\\
&\leq\frac{3}{4(1+\kappa_{\alpha(z)})}\cdot\norm{b_{n+1}-c_{n+1}}.
\end{aligned}\end{equation} By virtue of (\ref{F-whole-norm}), one has

\begin{equation}
\norm{b_n-c_n}<\frac{3}{4}\cdot\norm{b_{n+1}-c_{n+1}} \end{equation} It yields that $\lim\limits_{n\rightarrow+\infty}\norm{b_{n}-c_{n}}=+\infty$, a contradiction. Thus, we have proved this assertion.

We futher assert that $\alpha(z)$ is a singleton. By the assertion above, one can take a point $q\in\alpha(z)$ such that $\beta_q>0$. It entails that $\mid\lambda \mid>1$ for any $\lambda\in\sigma(D_q\Phi_{t_0}\mid_{E_q})$ with some $t_0>0$. It then follows from Lemma \ref{p-notin-omega.y} that one can find a $\tilde{\rho}>0$ such that for any $y\in X$ such that $y\thicksim q$ and $y\neq q$, there is a $T_y>0$ such that $\Phi_{T_y}(y)\in X\setminus B_{\tilde{\rho}}(q)$. By Lemma \ref{ordered trichotomy}, one has $\alpha(z)\thicksim\omega(x)$.  Together with $\alpha(z)\subset E$, one has that $\alpha(z)\cap B_{\tilde{\rho}}(q)=\{q\}$. Then, the connectedness of $\alpha(z)$ implies that $\alpha(z)=\{q\}$ is a singleton.

Finally, we prove that the alternative (c) does not happen. For any given $z\in \omega(x)\setminus \tilde{B}$, $\alpha(z)=\{q\}$ implies $\lim\limits_{t\rightarrow +\infty}\Phi_{-t}(z)=q$. By Lemma \ref{ordered trichotomy} and \ref{T-R-ordered trichotomy}, one has $\alpha(z)\thicksim\omega(x)$, and hence, $\{q\}\thicksim\omega(x)$. By the strongly focusing monotonicity of $\Phi_t$, one can find constants $\delta,T,\kappa>0$ such that there exists a strongly focusing operator $T_{(\tilde{\tilde{z}},q)}$ with its separation index greater than $\kappa$, and $\Phi_T(\tilde{\tilde{z}})-q=T_{(\tilde{\tilde{z}},q)}(\tilde{\tilde{z}}-q)$ for any $\tilde{\tilde{z}}\in \omega(x)\cap B_{\delta}(q)$. By repeating the arguments on (\ref{Mainsubspace-estimates}), one has 

$$\norm{\Pi^{E_q}(\Phi_{T}(\tilde{\tilde{z}})-q)}\geq \kappa\cdot\norm{\Phi_{T}(\tilde{\tilde{z}})-q}$$ for any $\tilde{\tilde{z}}\in\omega(x)\cap (B_{\delta}(q)\setminus\{q\})$. Since $\Phi_t$ admits a flow extension on $\omega(x)$, $\Phi_T$ is a homomorphism on $\omega(x)$. Then, there is a $\rho\in(0,\delta)$ such that $\Phi_{-T}(\tilde{\tilde{z}})\in B_{\delta}(q)$ for any $\tilde{\tilde{z}}\in B_{\rho}(q)$. Hence, 

\begin{equation}\label{local decom around q}\norm{\Pi^{E_q}(\tilde{\tilde{z}}-q)}\geq \kappa\norm{\tilde{\tilde{z}}-q}\end{equation} for any $\tilde{\tilde{z}}\in\omega(x)\cap (B_{\rho}(q)\setminus\{q\})$. The compactness of $\omega(x)$ implies that $\omega(x)\setminus(\text{Int}B_{\rho}(q))$ is compact. Together with $\{q\}\thicksim \omega(x)$, there is a $\tilde{\kappa}\in (0,\frac{\kappa}{2})$ such that  

\begin{equation}\label{decom whole omega}
\norm{\Pi^{E_q}(\tilde{\tilde{z}}-q)}\geq 2\tilde{\kappa}\norm{\tilde{\tilde{z}}-q}>\tilde{\kappa}\norm{\tilde{\tilde{z}}-q}\end{equation} for all $\tilde{\tilde{z}}\in\omega(x)\setminus\{q\}$. Let 

$$ \mathcal{N}=\{\tilde{\tilde{z}}\in X\setminus\{q\}:\norm{\Pi^{E_q}(\tilde{\tilde{z}}-q)}\geq \tilde{\kappa}\norm{\tilde{\tilde{z}}-q}\}.$$ 

By the strong monotonicity of $\Phi_t$ and $\omega(x)\setminus\{q\}\neq\emptyset$, there is a $t_1>0$ such that $\Phi_{t_1}(x)\thickapprox q$. It yields that $\Phi_{t}(x)\thickapprox q$ for all $t>t_1$. Recall that $\beta_{q}>0$. Then, there is $t_0>0$ such that $\mid \lambda\mid>1$ for any $\lambda\in \sigma(D_q\Phi_{t_0}\mid_{E_q})$. Here, we take the constant $\rho$ samller than the one mentioned in Lemma \ref{p-notin-omega.y}. Together with $\lim\limits_{t\rightarrow +\infty}d(\Phi_t(x),\omega(x))=0$, there is a $T_0>\max\{t_0,t_1\}$ such that 

\begin{equation}\label{Normal cone united with ball} \Phi_t(x)\in (B_{\rho}(q)\cup \mathcal{N})\cap B_{\rho}(\omega(x))\end{equation} for any $t>T_0$, where $B_{\rho}(\omega(x))=\{z\in X:\,d(z,\omega(x))\leq\rho\}.$  

By the smoothness of $\Phi_t$, the compactness of $\omega(x)$ and taking $\rho>0$ smaller if necessary, one has that 

\begin{equation}\label{nonlinear term-T} \norm{D_u\Phi_{\tau}-D_v\Phi_{\tau}}\leq \frac{ \tilde{k}}{2 M_{2,q}\norm{\Pi^{E_q}}}\end{equation} for any $u,v\in (B_{\rho}(q)\cup \mathcal{N})\cap B_{\rho}(\omega(x))$ and $\tau\in[0,T]$. Clearly, 

\begin{equation}\label{decom t>T_1}\norm{\Pi^{E_q}(\Phi_t(x)-q)}\geq \tilde{k}\norm{\Phi_t(x)-q}\end{equation} for any $\Phi_t(x)\in \mathcal{N}\cap B_{\rho}(\omega(x))$ with $t>T_0$. Then, we have the following estimate: 

\begin{equation}\label{Ine-tilde{T}+T-tilde{T}}\begin{aligned} \norm{\Pi^{E_q}(\Phi_{t+\tau}(x)-q)}&\geq\norm{D_q\Phi_{\tau}\circ\Pi^{E_q}(\Phi_t(x)-q)}\\
&\,\,\,\,\,\,-\norm{\int_0^1\Pi^{E_q}\circ[D_{q+s(\Phi_t(x)-q)}\Phi_{\tau}-D_q\Phi_{\tau}]ds
(\Phi_t(x)-q)}\\
&\,\,\,\,\,\,\overset{(\ref{eq:M-3})+(\ref{nonlinear term-T})}{\geq}\frac{e^{\beta_q \tau}}{M_{2,q}}\norm{\Pi^{E_q}(\Phi_t(x)-q)}-\frac{ \tilde{k}}{2 M_{2,q}\norm{\Pi^{E_q}}}\cdot\norm{\Pi^{E_q}}\norm{\Phi_t(x)-q}\\
&\,\,\,\,\,\,\overset{(\ref{decom t>T_1})}{\geq}\frac{\tilde{\kappa}\cdot (e^{\beta_q \tau}-\frac{1}{2})}{M_{2,q}\cdot\norm{\Pi^{E_q}}}\cdot\norm{\Pi^{E_q}}\cdot\norm{\Phi_t(x)-q}
\end{aligned}\end{equation} for any $\tau\in[0,T]$ and $\Phi_t(x)\in \mathcal{N}\cap B_{\rho}(\omega(x))$ with $t>T_0$. Let 

$$\kappa^{**}=\frac{\tilde{k}}{2 M_{2,q}\cdot \norm{\Pi^{E_q}}  }.$$ It then follows from (\ref{Ine-tilde{T}+T-tilde{T}}) that

\begin{equation}\label{rho''-estimates}\norm{\Phi_{t+\tau}(x)-q}\geq\frac{\norm{\Pi^{E_q}(\Phi_{t+\tau}(x)-q)}}{\norm{\Pi^{E_q}}}\geq \kappa^{**}\cdot\norm{\Phi_t(x)-q}
\end{equation} for any $\tau\in[0,T]$ and $\Phi_t(x)\in\mathcal{N}\cap B_{\rho}(\omega(x))$ with $t>T_0$. Let 

$$\rho^{\prime}=\min\{\kappa^* \kappa^{**},\kappa^{**}\}\cdot \frac{\rho}{2} \,(=\frac{\kappa^* \kappa^{**}}{2}\cdot \rho \leq\frac{\rho}{16}),$$ because of $\kappa\in(0,1)$, $\tilde{\kappa}\in(0,\frac{\kappa}{2})$ and $\norm{M_{2,q}},\norm{\Pi^{E_q}}\geq 1$. By utilizing $\lim\limits_{t\rightarrow +\infty}d(\Phi_t(x),\omega(x))=0$ again, there is a $T_1>T_0$ such that 

\begin{equation}\label{2-around omega n} \Phi_t(x)\in (B_{\rho^{\prime}}(q)\cup
\mathcal{N})\cap B_{\rho^{\prime}}(\omega(x))\end{equation} for any $t\geq T_1$. Thus, for any $t>T_1$, one has that:

\begin{equation}\label{Dichotomy}\text{If} \,\,\Phi_t(x)\notin B_{\rho^{\prime}}(q),\,\, \text{then}\,\,\Phi_t(x)\in B_{\rho^{\prime}}(\omega(x))\cap \mathcal{N}.\end{equation}
Clearly, for any $t>T_1$, if $\Phi_{t+s}(x)\in B_{\rho}(q) \big(\supsetneqq B_{\rho^{\prime}}(q)\big)$ for any $s\in[0,N\tilde{T}]$ with a non-negative integer $N$, then it follows from Lemma \ref{p-notin-omega.y} that (\ref{Lower-bound-growth-estimates}) holds with $N$, and becomes that for any non-negative integer $n$ with $n\leq N$, 

\begin{equation}\label{Lower-bound-growth-estimates-trans4.2-N}\begin{aligned}
&\quad \norm{\Phi_{t+n\tilde{T}+\tau}(x)-q}\geq 3^{n}\kappa^*\norm{\Phi_{t+T}(x)-q}\,\, \text{with}\,\,\tau\in[T,\tilde{T}+T],\\
&\text{and}\\
&\quad \norm{\Phi_{t+n\tilde{T}+\tau}(x)-q}\geq 3^{n+1}\norm{\Phi_{t+T}(x)-q}\,\, \text{with}\,\,\tau\in[\tilde{T},\tilde{T}+T].
\end{aligned}\end{equation} Notice that $q\in \omega(x)$. Thus, there is a $s_0>T_1$ such that $\Phi_{s_0}(x)\in B_{\rho^{\prime}}(q)$. Furthermore, (\ref{Dichotomy})-(\ref{Lower-bound-growth-estimates-trans4.2-N}) and the continuity of $\Phi_t$ imply that there is $s_1>s_0$ such that 

$$\begin{aligned}& \quad\Phi_{s_1}(x)\in \partial B_{\rho}(q)\cap B_{\rho^{\prime}}(\omega(x))\cap \mathcal{N}\\
&\text{and}\\ 
&\quad \Phi_{s_0+\tau}(x)\in B_{\rho}(q) \,\,\text{for any}\,\,\tau\in [0,s_1-s_0].\end{aligned}$$ It then entails that \begin{equation}\label{estimates T to tilde{T}+T}\begin{aligned}
&\quad\norm{\Phi_{s_1+\tau}(x)-q}\overset{(\ref{rho''-estimates})}{\geq} \kappa^{**}\norm{\Phi_{s_1}(x)-q}=\kappa^{**}\rho \geq 4\rho^{\prime}\quad \text{for any}\,\, \tau\in[0,T]\\
&\text{and}\\
&\quad\norm{\Phi_{s_1+\tau}(x)-q}\overset{(\ref{Lower-bound-growth-estimates-trans4.2-N})\,{\rm with}\,N=0}{\geq}\kappa^*\cdot\norm{\Phi_{s_1+T}(x)-q}\overset{(\ref{rho''-estimates})}{\geq}\kappa^*\cdot\kappa^{**}\norm{\Phi_{s_1}(x)-q}\\
&\quad\quad\quad\quad\quad\quad\quad\quad=\kappa^* \kappa^{**}\rho\geq 2\rho^{\prime}\quad \text{for any}\,\, \tau\in[T,\tilde{T}+T]. \end{aligned}\end{equation} Specially, 

\begin{equation}\label{estimates tilde{T} to tilde{T}+T}\begin{aligned}
&\norm{\Phi_{s_1+\tilde{T}+\tau}(x)-q}\overset{(\ref{Lower-bound-growth-estimates-trans4.2-N})\,{\rm with}\,N=0}{\geq}3\cdot\norm{\Phi_{s_1+T}(x)-q}\overset{(\ref{rho''-estimates})}{\geq}3\kappa^{**}\norm{\Phi_{s_1}(x)-q}\\
&\quad\quad\quad\quad\quad\quad\quad\quad=3\kappa^{**}\rho\geq 12\rho^{\prime} \quad \text{for any}\,\, \tau\in [0,T].\end{aligned}\end{equation} Suppose that there is a $s_2>s_1$ such that $\Phi_{s_2}(x)\in\partial B_{\rho}(q)\cap B_{\rho^{\prime}}(\omega(x))\cap\mathcal{N}$. It follows from the same arguments that (\ref{estimates T to tilde{T}+T})-(\ref{estimates tilde{T} to tilde{T}+T}) also holds for the point $\Phi_{s_2}(x)$. Clearly, $\rho\geq 16\rho^{\prime}$. (\ref{Lower-bound-growth-estimates-trans4.2-N}) holds for any $\Phi_{t}(x)$ satisfying $\Phi_t(x)\in B_{\rho}(q)$ with $t\geq T_1$. By the continuity of $\Phi_t$ again, one has that 

$$\Phi_t(x)\notin B_{\rho^{\prime}}(q)$$ for any $t>s_1$, a contradiction to $q\in \omega(x)$.

Therefore, we have completed the proof.
\end{proof}

Now, we prove main results.

\noindent{\it Proof of Theorem A: } Lemmas \ref{Lemma for ii} and \ref{Lemma for iii} directly imply that $\omega(x)$ is an ordered set. Since $C$ is a complemented $k$-solid cone, one can choose $k$-dim linear subspace $H$ and $k$ co-dimensional linear subspace $H^c$ such that $H\subset C$, $H^c\setminus\{0\}\subset X\setminus C$ and $H\oplus H^c=X$. Let $\Theta$ be the natural projection onto $H$ along $H^c$. Since $L$ is ordered, one has that $\Theta\mid_{L}$ is one-to-one. By repeating the argument in the proof of \cite[Theorem 3.4]{Smi95}, one has that $\omega(x)$ is topologically conjugate to a compact invariant set of a continuous vector field in $\mathbb{R}^k$.

\vskip 3mm
Thus, we have completed the proof of Theorem A. $\quad\quad\quad\quad\quad\quad\quad\quad\quad\quad\quad\quad\quad\quad\quad\quad\quad\quad\square$

\vskip 3mm
{\it Proof of Theorem B: } By virtue of \cite[Theorem A]{F}, the set $\text{Int}(Q\cap C_E)\cap \mathcal{D}$ is open and dense in $\mathcal{D}$. Clearly, for any $x\in C_E$, $\omega(x)$ is ordered. Together with  Theorem A, we obtain that for any point $x\in \text{Int}(Q\cap C_E)\cap \mathcal{D}$, the $\omega$-limit set $\omega(x)$ is ordered.

\vskip 3mm
Therefore, we have completed the proof. $\quad\quad\quad\quad\quad\quad\quad\quad\quad\quad\quad\quad\quad\quad\quad\quad\quad\quad\quad\quad\quad\square$

\vskip 3mm
{\it Proof of Theorem C: } By virtue of Theorem A, $\omega(x)$ is ordered. It should be pointed out that all arguments in the proof of \cite[Theorem B]{F-W-W4} are still effective for the semiflow strongly focusing monotone w.r.t. a $k$-cone on an infinite dimensional Banach space. So, the infinite dimensional version of the solid Poincar\'{e}-Bendixson theorem is established.

\vskip 3mm
Therefore, we have proved Theorem C. $\quad\quad\quad\quad\quad\quad\quad\quad\quad\quad\quad\quad\quad\quad\quad\quad\quad\quad\quad\quad\quad\quad\square$

\vskip 3mm
{\it Proof of Theorem D: } By Theorem B, the set $\text{Int}(Q\cap C_E)\cap \mathcal{D}$ is open and dense in $\mathcal{D}$. By virtue of Theorem C, for any $x\in Q$, $\omega(x)$ is either a nontrivial periodic orbit or a set consisting of equilibria with their potential connected orbits. Thus, we obtain that there is open dense set $\mathcal{D}^{\prime}\subset\mathcal{D}$ such that for any $x\in \mathcal{D}^{\prime}$, $\omega(x)$ is either a nontrivial periodic orbit or a set consisting of equilibria with their potential connected orbits.

\vskip 3mm
Therefore, we have completed the proof. $\quad\quad\quad\quad\quad\quad\quad\quad\quad\quad\quad\quad\quad\quad\quad\quad\quad\quad\quad\quad\quad\square$

\end{document}